\documentclass[article]{amsart}

\usepackage{amssymb,amsfonts,amsmath,amsthm}
\usepackage[all,arc]{xy}
\usepackage{enumerate}
\usepackage{mathrsfs}
\usepackage[toc,page]{appendix}
\usepackage[left=3cm, right=3cm, bottom=3cm]{geometry}
\usepackage{tabularx}
\usepackage{url}
\usepackage{color}
\usepackage{tikz-cd}  
\usepackage{mathdots} 
\usepackage{romannum} 

\newtheorem{thm}{Theorem}[section]

\newtheorem*{thm*}{Theorem}
\newtheorem*{cor*}{Corollary}
\newtheorem*{prop*}{Proposition}
\newtheorem{cor}[thm]{Corollary}
\newtheorem{prop}[thm]{Proposition}
\newtheorem{lem}[thm]{Lemma}

\theoremstyle{definition}
\newtheorem{defn}[thm]{Definition}

\newtheorem{con}[thm]{Construction}
\newtheorem{exmp}[thm]{Example}

\newtheorem*{notn*}{Notation}

\theoremstyle{remark}

\newtheorem*{idea*}{Idea}

\makeatletter
\let\c@equation\c@thm
\makeatother
\numberwithin{thm}{section}
\numberwithin{equation}{section}

\title[]{On a class of indefinite Kac-Moody algebras}

\author{Kehan Wang}

\begin{document}

\maketitle
\begin{abstract}
In this paper, we study a special class of indefinite Kac-Moody algebras. Based on the study of hyperbolic Kac-Moody algebras, we give the definition of $N_k$ type Kac-Moody algebras and study some properties of this special type Kac-Moody algebras.
\end{abstract}

\flushbottom

\section{Introduction}
\subsection{Background}
A Kac-Moody algebra is defined by a generalized Cartan matrix (GCM), which directly gives us the structure of the Kac-Moody algebra and its certain properties. In this paper, we construct and study new classes of generalized Cartan matrices, which correspond to indecomposable Kac-Moody algebras of the indefinite type over $\mathbb{C}$.

It is well-known that any indecomposable Kac-Moody algebra and the corresponding GCM, fall into one of the following three types: \emph{finite}, \emph{affine} or \emph{indefinite}. A finite Kac-Moody algebra is exactly a semisimple Lie algebra. There is a one-to-one correspondence between indecomposable finite Kac-Moody algebras and diagrams, which is a certain class of directed graphs \cite{Hum}. When it comes to the affine type, we have an one-to-one correspondence between indecomposable affine Kac-Moody algebras and extended diagrams \cite{Kac}. By the above discussion, there is a one-to-one correspondence between the following three objects in the finite and affine cases.
\begin{align*}
\text{Kac-Moody algebras } \Longleftrightarrow \text{ generalized Cartan matrices } \Longleftrightarrow \text{Diagrams }
\end{align*}

Here is another very interesting property about Kac-Moody algebras. An affine GCM is not finite, but any proper principal submatrix is of finite type. A GCM is \emph{hyperbolic} if it is neither finite nor affine, and any proper principal submatrix is either of finite or of affine type. Based on this property, we can extend the definition of hyperbolic GCM and define a new class of indefinite Kac-Moody algebras.

We make the following definitions of new types of Kac-Moody algebras (see Definition \ref{203} and Corollary \ref{204}).
\begin{enumerate}
    \item Type $N_0$: finite type GCM.
    \item Type $N_1$: affine type GCM.
    \item Type $N_2$: Hyperbolic type GCM.
    \item Type $N_k$: a GCM is not of type $N_{m}$ for $m < k$, but every proper principal submatrix is of type $N_m$ for $m<k$.
\end{enumerate}
Furthermore, we say that a GCM $A$ is of type $N_{k,n}$ if $A$ is of type $N_k$ with dimension $n$. The goal of this paper is to study these new classes of GCMs.

\subsection{Description of the main results}
In this paper, we study the methods of constructing and bounds of dimensions of $N_{k}$ type generalized Cartan matrices. We prove the following fundamental properties of the generalized Cartan matrices:

\begin{thm}[\textbf{Theorem \ref{301}}]
Given a $N_{k,n}$, $k>2$, type generalized Cartan matrix $A$, there exists some $N_{k-1,n-1}$ type generalized Cartan matrices as principal submatrices.
\end{thm}

This property tells us that every GCM of type $N_k$ contains a smaller GCM of type $N_{k-1}$. On the other hand, in the next theorem, we start from a smaller GCM and construct larger GCMs.

\begin{thm}(Theorem \ref{302})
For a $N_{n-1,n}$, $n\geq2$, type generalized Cartan matrix, by adding one more dimension, we can obtain a $N_{n,n+1}$ type generalized Cartan matrix if and only if entries $a_{nj}$ for $j\in \{1,2,...,n-1\}$ are in the following several cases:
\begin{enumerate}
\item $a_{nj} = a_{jn} = 0$;
\item $a_{nj}=a_{jn}=-1$;
\item $a_{nj}=-1$, and $a_{jn}=-2$, $-3$, or $-4$;
\item $a_{nj}$ = $a_{jn}$ = $- 2$,
\item not all entries are $0$.
\end{enumerate}
\end{thm}

Besides, we provide the upper and lower bounds of dimensions of generalized Cartan matrices. The bounds are varied and specific to the type of the generalized Cartan matrices. Specifically, we provide examples to prove the supremum and infimum dimensions of compact and non-compact $N_3$ type generalized Cartan matrices. These results will help us construct the generalized Cartan matrices and better understand the higher dimension generalized Cartan matrices.

\subsection{Structure of the paper}
In \S 2, we provide basic definitions of generalized Cartan matrices and the corresponding diagrams. In \S 3, we describe the general methods for constructing new types of generalized Cartan matrices. In \S 4, we discuss the upper and lower bounds of the dimensions of $N_k$ type generalized Cartan matrices. Specifically, we provide the supremum of dimensions of compact and non-compact $N_{3}$ type generalized Cartan matrices.

\section{definition}

A matrix $A=(a_{ij})_{1 \leq i,j \leq \ell}$ is a \textit{generalized Cartan matrix}, if it satisfies the following conditions,
\begin{enumerate}
\item the entries of $A$ are integers;
\item $a_{ii}=2$;
\item $a_{ij}<0$, $i\neq j$;
\item $a_{ij}=0$ then $a_{ji}=0$.
\end{enumerate}

Correspondingly, for a generalized Cartan matrix $A$, a \textit{Dynkin diagram} $D=(V,E)$ can be constructed. The set of vertices $V = \{v_{1}, \cdots , v_{\ell}\}$ corresponds to the columns (or rows) of $A$, and the set of edges $E_{ij}$ between $v_{i}$ and $v_{j}$ depends on the entries $a_{ij}$, and $a_{ji}$ of $A$. The edges between $v_{i}$ and $v_{j}$ can be characterized as follows:
\begin{enumerate}
\item No edge: if $a_{ij} = a_{ji} = 0$;
\item Single edge: if $a_{ij}=a_{ji}=-1$;
\item Directed arrow with double, triple or quadruple edges (asymmetric): if $a_{ij}$ = $-1$, and $a_{ji}$  = $-2, -3, or -4$;
\item Double edges with double-headed arrow (symmetric):if $a_{ij}$ = $a_{ji}$ = $- 2$;
\item Labeled edge (symmetric or asymmetric): if it is none of the above types, and $a_{ij}$ = $- a$, $a_{ji}$ = $- b$, where $a, b\in \mathbb{Z}_{+}$.
\end{enumerate}

Meanwhile we define the \emph{multiplicity} as $\text{mult}(i,j):=a_{ij}/a_{ji}$.

Based on the research done on the generalized Cartan matrices, they are classified into finite, affine, indefinite types. Furthermore, we can separate hyperbolic type out of the indefinite type of generalized Cartan matrices.

\begin{defn}\label{201}
Let $A$ is a a real $n\times n$ generalized Cartan matrix. Then one and only one of the following possibilities hold for both $A$ and $A^{t}$:
\begin{enumerate}
\item \textit{Finite}, det$A\neq 0$; there exists $u>0$ such that $Au>0$; $Au>0$ implies $u>0$ or $u=0$;
\item \textit{Affine}, corank $A=1$; there exists $ u>0$ such that $Au=0$ ; $Au\geq 0$ implies $Au=0$;
\item \textit{Indefinite}, there exists $u>0$ such that $Au<0$; $Au\geq 0$, $u\geq 0$ implies $u=0$;
\item \textit{hyperbolic}, $D$ is neither of finite nor affine type, but every proper connected subdiagram is either of finite or of affine type.
\end{enumerate}
\end{defn}

In the definition, the notation $Au>0$ means that every entry in $Au$ is larger than $0$, $Au=0$ means that every entry in $Au$ is $0$, and $Au<0$ means that every entry in $Au$ is smaller than $0$. $u>0$, $u=0$, $u<0$ are similar.

The definitions of finite, affine and indefinite Kac-Moody algebra follow from \cite[Theorem 4.3]{Kac} which is equivalent.

From the perspective of root decomposition, we only focus on the symmetrizable indecomposable indefinite generalized Cartan matrix in this paper.

For a generalized Cartan matrix $A$ to be \emph{indecomposable}, $A$ cannot be written as a block diagonal matrix, $A = diag(A_{1},A_{2})$ up to reordering of the rows and columns, where $A_{1}$, $A_{2}$ are generalized Cartan matrices. Correspondingly, the diagram should be connected.

Meanwhile, for a generalized Cartan matrix $A$ to be\textit{ symmetrizable}, there exist nonzero rational numbers $d_{1},...,d_{\ell}$, where $D=diag(d_{1},...,d_{\ell})$, such that the matrix $DA$ is symmetric. The matrix $DA$  is called a \textit{symmetrization} of $A$.

Now based on the definition of the symmetrisability, we have the following proposition.
\begin{prop}[Proposition 2.2 in \cite{Carb}]\label{202}
Let $A=(a_{ij})_{1 \leq i, j\leq \ell}$ be a generalized Cartan matrix. Then A is symmetrizable if and only if $a_{i_{1}i_{2}}a_{i_{2}i_{3}}...a_{i_{k}i_{1}}=a_{i_{2}i_{1}}a_{i_{3}i_{2}}...a_{i_{1}i_{k}}$
for each $i_{1},...,i_{k}\in\{1,...,\ell\}$, $k\leq \ell$, $i_{s}\neq i_{s+1}$ for $s$ ${\rm mod}$ $k$.
\end{prop}

Now based on the finite, affine and indefinite types of generalized Cartan matrices, we use the idea of mathematical induction to further classify the indefinite type of generalized Cartan matrices.

Following the definition of the hyperbolic type, we consider separating out other types of generalized Cartan matrices from the indefinite type.

\begin{defn}\label{203}
The set of generalized Cartan matrices of $N_{0}$ type is the set of the matrices of finite type. The set of generalized Cartan matrices of $N_{1}$ type is the set of the matrices of affine type. The set of generalized Cartan matrices of $N_{2}$ type is the set of the matrices of hyperbolic type.
A matrix $A $ is called of \textit{$N_{k}$  type}, $k \geq 3$, if it contains at least one principal submatrix of $N_{k-1}$ type, and every other principal submatrix is of $N_{m}$ type where $m<k$.
\end{defn}

Note that a hyperbolic generalized Cartan matrix is neither of finite nor affine type, but every proper connected submatrix is either of finite or of affine type.

\begin{cor}\label{204}
A generalized Cartan matrix $A$ is of type $N_k$, $k>1$, if $A$ is not of type $N_m$ for all $m < k$, but every proper connected submatrix is of type $N_{m'}$ for some $0 \geq m' < k$.
\end{cor}

\begin{proof}
The equivalence relation can be proved by mathematical induction on $k$ in $N_{k}$. Firstly, it is obvious that both definitions stand for the $N_{3}$ type. Now, we assume they are equivalent for $N_{k}$ type where $3<k\in\mathbb{Z}$. Then based on Definition \ref{203} for $N_{k+1}$ type, if the generalized Cartan matrix $A$ has $N_{k}$ type principal submatrices, then where $m<k+1$, we have $A$ is not of $N_{m}$ type by alternate definition of $N_{m}$ type, where $m<k+1$. Since the second part of the two definitions are the same, we have the equivalence relation of the two definitions.

Conversely, based on alternate definition for $N_{k+1}$ type, if generalized Cartan matrix $A$ is not of type $N_m$ for all $m < k+1$, then by Definition \ref{203} of $N_{m}$ type, where $m<k+1$, we have $A$ is of some $N_{m}$ type. Hence $A$ must contain some $N_{k-1}$ type principal submatrix. Therefore, two definitions are equivalent for $N_{k+1}$ type. Conclusively, the definitions are equivalent for all $k>3$.
\end{proof}
Now for a generalized Cartan matrix, we consider its \textit{dimension} as the number of columns (or rows) of the matrix. We denote $N_{k}$ type with dimension $n$ as $N_{k,n}$ type. Meanwhile, we define the principal submatrices of a $N_{k,n}$ type generalized Cartan matrix as the submatrices of some type $N_{k_{i},n_{i}}$ type, where $k_{i}<k$, $n_{i}<n$.

Due the fact that hyperbolic type generalized Cartan matrices does not necessarily contain affine type principal submatrices, we say a generalized Cartan matrix is of \textit{compact type} if no principal submatrix is of affine type.

\section{$N_k$-type Indefinite Kac-Moody Algebras}

In the section, we provide two ways of constructions for $N_{k+1,n+1}$ type generalized Cartan matrices based on the $N_{k,n}$ type generalized Cartan matrices.

\subsection{Inductive Method}
The key idea of the first construction is based on the idea of mathematical inductions. Before we introduce the construction and the proof to the feasibility, we give the following theorem which gives reference to why we can constructions based on a generalized Cartan matrix.

\begin{thm}\label{301}
Given a $N_{k,n}$, $k>2$, type generalized Cartan matrix $A$, there exists some $N_{k-1,n-1}$ type generalized Cartan matrices as principal submatrices.
\end{thm}

\begin{proof}
Let $A$ be a $N_{k,n}$, $k\geq 2$, type generalized Cartan matrix. By Definition \ref{203}, the matrix $A$ must contain some $N_{k-1}$ type  principal submatrices. Let the dimension of the $N_{k-1}$ type principal submatrix to be some $\ell\in \mathbb{Z}^+$. If $\ell<n-1$, then adding one more dimension to the matrix gives a generalized Cartan matrix containing $N_{k-1}$ type generalized Cartan matrices as principal submatrices. By Corollary \ref{204}, the new generalized Cartan matrix must be of type $N_{k-1+m,\ell +1}$, where $m\geq 1$ and $m>1$ if and only if there are principal submatrices of type higher than $N_{k-1}$. This contradicts to Definition \ref{203} of $N_{k}$ type generalized Cartan matrix. Hence the $N_{k-1}$ type principal submatrices can only have dimension $n-1$.
\end{proof}

The proof to Theorem \ref{301} also gives proof to the following corollary.

\begin{cor}\label{3001}
For a $N_{k,n}$, $k\geq 2$, type generalized Cartan matrix, the $N_{k-1}$ type principal submatrices have dimension at least $n-1$.
\end{cor}
This is important since it helps to verify the matrices constructed.

Because all $N_{k,n}$ type generalized Cartan matrices can be constructed based on $N_{k-1,n-1}$ type generalized Cartan matrices, we consider the algorithm constructing $N_{k+1,n+1}$ type generalized Cartan matrix as adding one more dimension on a $N_{k,n}$ type generalized Cartan matrix as following theorem,

\begin{thm}\label{302}
For a $N_{n-1,n}$, $n\geq2$, type generalized Cartan matrix, by adding one more dimension, we can obtain a $N_{n,n+1}$ type generalized Cartan matrix if and only if entries $a_{nj}$ for $j\in \{1,2,...,n-1\}$ are in the following several cases:
\begin{enumerate}
\item $a_{nj} = a_{jn} = 0$;
\item $a_{nj}=a_{jn}=-1$;
\item $a_{nj}=-1$, and $a_{jn}=-2$, $-3$, or $-4$;
\item $a_{nj}$ = $a_{jn}$ = $- 2$;
\item not all entries are $0$.
\end{enumerate}
\end{thm}

\begin{proof}
Note that $a_{nj}$, $a_{jn}$ forms a $N_{2,2}$ type principal submatrix if and only if it does not satisfy the above 4 cases. Assuming there is some $N_{n,n+1}$ type generalized Cartan matrix with a $N_{2,2}$ type principal submatrix, then adding one more dimension by Definition \ref{203} would give a $N_{3,3}$ type. Subsequently, by repeating the same process we can only get a $N_{n+1,n+1}$ type. This is contradictory to our assumption of a $N_{n,n+1}$ type generalized Cartan matrix. Hence all the dimension $2$ principal submatrices are either of finite or affine type. So the forward statement is true.

Conversely, by adding one more dimension following the requirements, all the dimension $2$ principal submatrices are either of finite type or affine type. The generalized Cartan matrix can not be of $N_{n-1,n+1}$ type by Corollary \ref{204}, since if it contains a $N_{n-1,n}$ type principal submatrix. By Theorem \ref{301}, the $N_{k,n+1}$ type, $k>n$, generalized Cartan matrix would contain a $N_{k-1,n}$ type principal submatrix. Repeating the same process, the $N_{k,n+1}$ type generalized Cartan matrix contains a $N_{2,2}$ type principal submatrix. This is contradictory to the fact that all principal submatrices are of finite or affine type. So it can not be of $N_{k,n+1}$, for some $k>n$ type. Hence the new generalized Cartan matrix can only be of $N_{n,n+1}$ type. So the backward statement is true.
\end{proof}

Observing the above ways of construction, we have the following corollary.

\begin{cor}\label{3002}
Given a generalized Cartan matrix $A$ of type $N_{k,k+a}$, $k\geq2$ and $a\in \mathbb{Z}^+$, all principal submatrices of $A$ with dimension smaller than $a+1$ are finite. The lowest dimension for the principal submatrices of hyperbolic type is $a+2$.
\end{cor}

\begin{proof}
We prove the corollary by inductions on $k,a$ in $N_{k,k+a}$. When $k=2$, $a=1$, the statement is true by Theorem \ref{302}.

We assume the statement is true for $k=n\geq 2$, $a$ as some constant. Now we consider $k=n+1$. We know the $N_{n+1,n+a+1}$ type generalized Cartan matrices contain a $N_{n,n+a}$ type principal submatrix by Theorem \ref{301}. Furthermore, by Corollary \ref{3001}, we have that the lowest dimension of the $N_{n}$ type generalized Cartan matrices is $n+a$. Hence by the inductive definition, principal submatrices with dimension smaller than $a+1$ are finite. The lowest dimension for the principal submatrices of hyperbolic type is $a+2$. Therefore the statement is true for $k=n+1$. By mathematical induction, it is true for all $k\geq 2$ with fixed $a$.

We assume the statement is true for $k=n\geq 2$, $a=m>0$. Now we consider $N_{n,n+m+1}$ type generalized Cartan matrices when $a=m+1$. By Theorem \ref{301}, the $N_{n,n+m+1}$ type generalized Cartan matrix would contain a $N_{n-1,n+m}$ type principal submatrix. Repeating the same process, the $N_{n,n+m+1}$ type generalized Cartan matrix contains a $N_{2,m+n+1-n+2}$ type principal submatrix where $m+n+1-n+2=m+3=a+2$. By Corollary \ref{204}, the principal submatrices with dimension smaller than $m+2$ are either of affine or finite type. If the hyperbolic type is compact, the principal submatrices with dimension smaller than $m+2$ are of finite type. If it is non compact, then Theorem \ref{301} gives the result that affine type principal submatrices can only have dimension $a+1$. So principal submatrices with dimension smaller than $a+1$ are always finite. Therefore the statement is true for $a=m+1$. By mathematical induction, it is true for positive integer $a$.

Hence the corollary is true for all $k\geq2$ and $a\in \mathbb{Z}^+$.
\end{proof}

Furthermore by Theorem \ref{301}, each $N_{k,n}$, $k>2$, type generalized Cartan matrix has some $N_{k-1,n-1}$ type principal submatrices. Hence every $N_{k,n}$ type generalized Cartan matrix can be obtained by adding one more dimension to a $N_{k-1,n-1}$ type generalized Cartan matrix. We will discuss the ways of adding dimensions by giving examples in section $4$.

However, one drawback of the above algorithm is that we have to check the symmetrizability of the generalized Cartan matrices constructed. It could be troublesome. Hence we have the modified algorithm which ensures the symmetrizability of the generalized Cartan matrix constructed.

\subsection{Modified Method}
Before the construction, we prove the following lemma about the symmetrizability.

\begin{prop}\label{303}
For a generalized Cartan matrix $A$, it is symmetrizable if and only if each principal submatrix is symmetrizable.
\end{prop}

\begin{proof}
For a generalized Cartan matrix $A=(a_{ij})_{1 \leq i,j \leq \ell}$, proving being symmetrizable is equivalent to proving
\begin{align*}
    a_{k_{1}k_{2}}a_{k_{2}k_{3}}...a_{k_{m}k_{1}}=a_{k_{2}k_{1}}a_{k_{3}k_{2}}...a_{k_{1}k_{m}}
\end{align*}
for any chosen positive integer $m\leq \ell$, where $k_{1},...,k_{m}\in\{1,...,\ell\}$, $k_{s}\neq k_{s+1}$ for $s$ ${\rm mod}$ $m$.

If $m=\ell-1$. Let $n$ be a positive integer such that $n\leq \ell$. We remove the $n$-th column and the $n$-th row, and obtain the submatrix $A'_{n}=(a_{ij})$, where $i_{n},j_{n}\in \{1,2,3,...,\ell\}/\{n\}$. Applying Proposition \ref{202}, we get
\begin{align*}
a_{k'_{1}k'_{2}}...a_{k'_{\ell-1}k'_{1}}=a_{k'_{2}k'_{1}}...a_{k'_{1}k'_{\ell-1}},
\end{align*}
where $k'_{1},...,k'_{\ell-1}\in\{1,...,\ell\}/\{n\}$, $k'_{s}\neq k'_{s+1}$ for $s$ ${\rm mod}$ $\ell-1$. For each positive integer $n$, we obtain a principal submatrix which will give us a similar equation as above. We time up all the right hand side and left hand side of all the equalities obtained from each principal submatrix. We have
\begin{align*}
(a_{k_{1}k_{2}}a_{k_{2}k_{3}}...a_{k_{\ell}k_{1}})^{(\ell-1)}=(a_{k_{2}k_{1}}a_{k_{3}k_{2}}...a_{k_{1}k_{\ell}})^{(\ell-1)}
\end{align*}
By Definition \ref{201}, we know that $a_{k_{i}k_{j}}\leq0$, if $i\neq j$. Therefore the left and right hand sides are both positive, negative or zero. By $(\ell-1)$-th root on each side, we obtain
\begin{align*}
a_{k_{1}k_{2}}a_{k_{2}k_{3}}...a_{k_{\ell}k_{1}}=a_{k_{2}k_{1}}a_{k_{3}k_{2}}...a_{k_{1}k_{\ell}}.
\end{align*}
Similarly we can obtain the equations for other chosen $m$. Therefore, the generalized Cartan matrix is symmetrizable by Proposition \ref{202}.

To prove the converse, if the  $A=(a_{ij})_{1\leq i,j\leq \ell}$ is a symmetrizable generalized Cartan matrix, we can fix a positive integer $m\leq \ell$. By Proposition $2.7$, $k_{1},...,k_{m}\in \{1,2,...,\ell\}$, $m\leq \ell$, $k_{s}\neq k_{s+1}$ for $s$ ${\rm mod}$ $m$ we have
\begin{align*}
a_{k_{1}k_{2}}a_{k_{2}k_{3}}...a_{k_{m}k_{1}}=a_{k_{2}k_{1}}a_{k_{3}k_{2}}...a_{k_{1}k_{m}}.
\end{align*}
Therefore each principal submatrix is symmetrizable.
\end{proof}

By the above statement, we can now ensure the symmetrizability of the newly constructed generalized Cartan matrix by simply ensuring that each principal submatrix is symmetrizable. Therefore, we can have the following construction.

\begin{con}\label{304}
Based on a $N_{k,k+a}$ type Cartan matrix $A$, $k\geq2$, $a\geq 1$, we can obtain a $N_{k+1,k+a+1}$ Cartan matrix by the following steps:
\begin{enumerate}
\item Write out the symmetrisation of $A$=$DA'$;
\item Add one more column and row into $D$;
\item Consider possible ways of putting entries into $A$.
\end{enumerate}
\end{con}

To show how it works, we give an example following the construction of generalized Cartan matrices using the above algorithm.

\begin{exmp}\label{305}
We consider the following $N_{2,4}$ type generalized Cartan matrix $A_{2}$,
\[
A_{2}=
  \begin{bmatrix}
    2&-2&0&-1\\
    -1&2&-1&0\\
    0&-1&2&-1\\
    -1&0&-2&2\\
  \end{bmatrix}
\]
This matrix $A_2$ is symmetrizable, and the diagonal matrix is $D_{2}$,
\[
D_{2}=
  \begin{bmatrix}
    1&0&0&0\\
    0&2&0&0\\
    0&0&2&0\\
    0&0&0&1\\
  \end{bmatrix}
\]

By Corollary \ref{3002}, all the dimension $2$ principal submatrix should be of $N_{0,2}$ type. Therefore with the new entry $d_{5}$ added in $D_{2}$, we have $d_i/d_j=E_{ij}\in\{1,2,3,1/2,1/3\}$ or $E_{ij}=0$, where each one is in representation of the possible value of $a_{12}/a_{21}$ in the $N_{0,2}$ type generalized Cartan matrix.

So we have the following requirements:
\begin{enumerate}
    \item When $d_{5}$ is $2$, we have
    \begin{align*}
        (a_{51},a_{15}),(a_{54},a_{45})\in\{(-1,-2),(0,0)\}, \quad (a_{53},a_{35}),(a_{52},a_{25})\in\{(0,0),(-1,-1)\},
    \end{align*}
    \item When $d_{5}$ is $1$, we have
    \begin{align*}
    (a_{51},a_{15}),(a_{54},a_{45})\in\{(-1,-1),(0,0)\}, \quad (a_{52},a_{25}),(a_{53},a_{35})\in\{(0,0),(-2,-1)\},
    \end{align*}
    \item When $d_{5}=3$, we have $a_{35}$, $a_{53}$, $a_{25}$ and $a_{52}$ to be $0$, and
    \begin{align*}
    (a_{54},a_{45}),(a_{51},a_{15})\in\{(-1,-3),(0,0)\},
    \end{align*}
    \item When $d_{5}$ is $6$, we have
    \begin{align*}
    (a_{51},a_{15}),(a_{54},a_{45})\in\{(0,0)\},\quad (a_{53},a_{35}),(a_{52},a_{25})\in\{(-1,-3)\},
    \end{align*}
    \item When $d_{5}$ is $1/2$, we have
    \begin{align*}
    (a_{52},a_{25}),(a_{53},a_{35})\in\{(-2,-1),(0,0)\}, \quad (a_{51},a_{15}),(a_{54},a_{45})\in\{(0,0),(-1,-1)\},
    \end{align*}
    \item When $d_{4}=1/3$, then it forces the entries $a_{35}$, $a_{53}$, $a_{25}$ and $a_{52}$ to be $0$, and
    \begin{align*}
    (a_{54},a_{45}),(a_{51},a_{15})\in\{(-3,-1),(0,0)\}.
    \end{align*}
\end{enumerate}

The matrices satisfying all the conditions above are the $N_{3,5}$ type generalized Cartan matrices constructed based on $A_{2}$.
\end{exmp}

In Construction \ref{304}, we observe that there are some rules in adding the entries. This helps to save the time of constructions. The first rule is as followed: For a generalized Cartan matrix $A=(a_{ij})$, with the symmetrization $DA$, $D=diag(d_{i})$. Then either $a_{ij}=0$, or $a_{ij}=k\cdot(d_{j}/d_{i})$, for some $k\in \mathbb{Z}$.

We know that the matrix $DA=(a'_{ij})$ is symmeftric, which implies $d_{i}a_{ij}=d_{j}a_{ji}$. The irst rule can be proved by considering the following two situations. If $a\neq 0$, we have $d_{i}a_{ij}=d_{j}a_{ji}\neq0$. Therefore we have $a_{ij}=a_{ji}(d_{j}/d_{i})$, $k=a_{ji}$. This finishes the proof.

As for the second rule, it also consider the entries in the special case of when $d_{\ell+1}$ are not both $1$: For a generalized Cartan matrix $A$, with the symmetrization $DA$. If $D=diag(d_{i})$, $i\in\{1,2,...,\ell\}$, and the new entry to $D$ is $d_{\ell+1}$. Then if $(d_{i},d_{\ell+1})= 1$, and $d_{i}$, for $i\in\{1,2,...,\ell\}$, $d_{\ell+1}$ are not both $1$, we have $a_{\ell+1,i}=a_{i,\ell+1}=0$ or $a_{i,\ell+1}=b\cdot d_{i}$ and $a_{\ell+1,i}=c\cdot d_{\ell+1}$ for some $b,c\in \mathbb{Z}$.

To prove the second rules, we mainly consider when $a_{\ell+1,i}\neq0$ since if $a_{\ell+1,i}=0$, the conclusion is trivial. So we assume that $a_{\ell+1,i}\neq0$. With the same calculation as in the rules concluded from Construction \ref{304}, we have $a_{\ell+1,i}=a_{i,\ell+1}(d_{i}/d_{\ell+1})$ which is an integer. Since $d_{i}$ and $d_{\ell+1}$ are coprime, we have $a_{\ell+1,i}=b\cdot d_{\ell+1}$ for some integer $b$. The same calculation holds for $a_{i,\ell+1}$.

\section{Boundedness of Dimensions of $N_k$ type Generalized Cartan Matrices}

In this section, we mainly discuss the maximum and minimum dimensions of generalized Cartan matrices. To begin with, we give the observation of the existence of such maximum and minimum dimensions.

\subsection{General Case}

\begin{lem}
The maximum dimension for a $N_{k}$, $k\geq2$, type generalized Cartan matrix is at most $8+k$.
\end{lem}

\begin{proof}
We prove the lemma by mathematical induction on $k$ in the $N_{k}$ type generalized Cartan matrices.

The maximum dimension for hyperbolic type generalized Cartan matrices is $10$. So the statement is true for $k=2$.

Now we assume the maximum dimension for $N_{n}$ type is $8+n$, for some $n\geq2$, and we consider the case when $k=n+1$. By Theorem \ref{301}, all the $N_{n+1}$ type generalized Cartan matrices can be obtained by adding one more dimension on some $N_{n}$ type generalized Cartan matrices. Therefore, the maximum dimension for $N_{n+1}$ type generalized Cartan matrices is at most $8+(n+1)$. So the statement is true for $k=n+1$.

By mathematical induction, we know that the statement is true for all positive integer $k>1$.
\end{proof}

As for the existence of the minimum dimension for the generalized Cartan matrices, it is obvious since dimensions can only be positive integers.

\begin{prop}
The minimum dimension of $N_{k}$, $k\geq2$, type generalized Cartan matrices is $k$.
\end{prop}
\begin{proof}
We prove the proposition by mathematical induction on $k$ in the $N_{k}$ type generalized Cartan matrices.

The minimum dimension for hyperbolic type generalized Cartan matrices is $2$. So the statement is true for $k=2$.

Now we assume the minimum dimension for $N_{n}$ type generalized Cartan matrices is $n$, for some $n\geq2$, and we consider the case when $k=n+1$. By the Theorem \ref{301} in section $3$, all $N_{n+1}$ type generalized Cartan matrices can be obtained by adding one more dimension on $N_{n}$ type generalized Cartan matrices. So the minimum dimension for $N_{n+1}$ type generalized Cartan matrices is at least $n+1$. If by adding one more dimension, we obtain a $N_{n+\ell,n+1}$ type generalized Cartan matrix, where $\ell>1$. Then by Theorem \ref{301}, there exists a $N_{n+\ell-1,n+1}$ type principal submatrix, and a $N_{n+\ell-1,n}$ type generalized Cartan matrix contain a $N_{n+\ell-2,n-1}$ type principal submatrix. Repeatedly, we have the $N_{n+\ell,n+1}$ type generalized Cartan matrix contains a $N_{n,n+1-\ell}$ type principal submatrix, where $\ell>1$ and $n+1-\ell<n$. This contradicts to our assumption that minimum dimension for $N_{n}$ type generalized Cartan matrices is $n$. Therefore the minimum dimension is $n+1$ for $k=n+1$.

By mathematical induction, we know that the statement is true for all positive integer $k>1$.
\end{proof}

Now before we begin our discussion about the more specific maximum dimension of the generalized Cartan matrices, we consider one observation important in the following discussions.

\begin{prop}
If we can obtain a $N_{k',n'}$ generalized Cartan matrix $A'=(a'_{ij})$ from some $N_{k,n}$ generalized Cartan matrix $A=(a_{ij})$ by exchanging columns and rows, then $k=k'$, $n=n'$.
\end{prop}

\begin{proof}
Since exchanging columns or rows do not change the number of dimensions the matrix has, to obtain $A'$ from $A$, $n$ must equal to $n'$.

Based on Definition \ref{201}, we have the type of a generalized Cartan matrices is related to its possible products with column matrices. Without loss of generality, we consider $\sum\limits^{n}_{i=1} a_{ij}v_{i}$, which is the product of $A=(a_{ij})$ with some matrix having the $i$th column vectors as $v_{i}$. Then when exchanging rows $i_{a}$ and $i_{b}$ in $A$, we can exchange $v_{a}$ and $v_{b}$ to give same sum. When exchanging the columns $j_{a}$ and $j_{b}$ in $A$, we have the sum remain the same. Therefore $k=k'$.
\end{proof}

\subsection{Maximal dimension of $N_3$-type Generalized Cartan Matrices}

In \S 4.1, we study the upper bound and lower bound of dimensions of $N_k$ type generalized Cartan matrices, in which the minimum plays as infimum for all $N_{k}$ type generalized Cartan matrices. So in this subsection, we study the supremum of dimensions in the case of compact and non-compact $N_3$ type separately. We do this through Construction \ref{304} on the principal submatrices starting from the highest-dimension ones.

We introduce the notations for the principal submatrices as follows. Given a $N_{2,n}$ type matrix $A$, let $A'$ be a principal submatrix of dimension $n-1$. We add one more row and column to $A$ and get a new matrix $A_1$. Then, the principal submatrix $A'$ will give us a dimension $n$ submatrix $A'_1$. We will construct $N_3$-type generalized Cartan matrix with respect to the matrix $A$ and $A'$. We will use the notation $(A,A')$ as the starting point.

Now by Theorem \ref{301}, all $N_{3,n}$ type generalized Cartan matrix would contain a $N_{2,n-1}$ type principal submatrix. Conversely, all $N_{3,n}$ type generalized Cartan matrices can be constructed from a $N_{2,n-1}$ type generalized Cartan matrix following Construction \ref{304}. By adding a dimension to $A$, we also add a dimension to $A'$. Note that the type of $A'$ could only be finite or affine. It means that constructing a $N_{3,n}$ type generalized Cartan matrix $A_1$ from $A$ is equivalent to construct a hyperbolic matrix $A'_1$ from $A'$. Therefore, we can construct all $N_{3,n}$ type generalized Cartan matrices with a careful discussion of the principal submatrix $A'$.

Before entering the discussion about the principal submatrices, we first introduce the notations.
Let $\mathcal{H}_{h_0,h_1,h_2}$ be the hyperbolic generalized Cartan matrix, which is corresponding matrix of dimension $h_2$ and in the $h_1$th-row of Table $h_{0}$ in \cite[Table 1, page 3759; Table 2, page 3762]{Sac}. Since the type of principal submatrix may be finite or affine, for those we follow the notations of the corresponding generalized Cartan matrices for diagrams in \cite[Chapter 4: Table aff1, Table aff2, Table aff3, Table fin]{Kac}. For example, $(\mathcal{H}_{1,2,3},B_2)$ is the hyperbolic matrix with principal submatrix $B_2$

\begin{align*}
\mathcal{H}_{1,2,3}=
  \begin{bmatrix}
    2&-1&0\\
    -3&2&-3\\
    0&-1&2\\
  \end{bmatrix},
  \quad
  B_{2}=
  \begin{bmatrix}
    2&-1\\
    -3&2\\
  \end{bmatrix}.
\end{align*}

We divide the discussion into two cases: the compact $N_{3}$ type and non-compact $N_{3}$ type generalized Cartan matrices. For each case, we initiate by choosing the hyperbolic type generalized Cartan matrices with highest dimension in that type and considering all the possible ways of adding a column and a row to obtain $N_{3}$ type generalized Cartan matrices.

We first consider the compact case. We will show that the maximal dimension is $5$ by constructing a $N_3$ type matrix $A_1$ with dimension $5$. As discussed above, we first take a pair $(\mathcal{H}_{1,4,4},B_{3})$, where

\begin{align*}
\mathcal{H}_{1,4,4}=
  \begin{bmatrix}
    2&-2&0&-1\\
    -1&2&-1&0\\
    0&-1&2&-1\\
    -1&0&-2&2\\
  \end{bmatrix},
  \quad
  B_{3}=
  \begin{bmatrix}
    2&-1&0\\
    -1&2&-1\\
    0&-2&2\\
  \end{bmatrix}.
\end{align*}

We start with the principal submatrix $A'=B_{3}$, and we will add one row and column to construct a new generalized Cartan matrix $A'_1$, which is considered as a principal submatrix of the expected $N_3$ type matrix $A_1$. Note that the type of $A'_1$ could only be compact hyperbolic type or finite type by Corollary \ref{204}. Therefore, the possible choices of $A'_1$ are $B_{4}$ and $F_{4}$. Take $A'_1=B_4$, we have
\begin{align*}
B_{4}=
 \begin{bmatrix}
    2&-1&0&0\\
    -1&2&-1&0\\
    0&-1&2&-1\\
    0&0&-2&2\\
 \end{bmatrix}.
\end{align*}
And we can let the $(a_{1,4}/a_{4,1})=0$, then we have the matrix
\begin{align*}
A_{1}=
 \begin{bmatrix}
     2&-2&0&-1&0\\
    -1&2&-1&0&0\\
    0&-1&2&-1&-1\\
    -1&0&-2&2&0\\
    0&0&-1&0&2\\
 \end{bmatrix}.
\end{align*}
The matrix $A_{1}$ is not of affine, finite nor hyperbolic type, but it has an dimension $4$ hyperbolic type principal submatrix and contains no affine type principal submatrices. So the maximum dimension of compact $N_{3,5}$ type generalized Cartan matrices is $5$.

As for non-compact cases, we will see that the maximal dimension of $N_3$ type cannot be $11$. The approach is that we choose all hyperbolic type matrices with dimension $10$ and try to construct new type $N_3$ type matrix with dimension $11$ with respect to the rule in Corollary \ref{204}. Subsequently, we will also show that $10$ is the supremum by giving examples of $N_{3}$ type generalized Cartan matrices of dimension $10$.

To begin with, we consider $(H_{2,79,10},B^{1}_{9})$.

\begin{align*}
H_{2,79,10}=
  \begin{bmatrix}
    2&-1&0&0&0&0&0&0&0&0\\
    -1&2&-1&0&0&0&0&0&0&0\\
    0&-1&2&-1&-1&0&0&0&0&0\\
    0&0&-1&2&0&0&0&0&0&0\\
    0&0&-1&0&2&-1&0&0&0&0\\
    0&0&0&0&-1&2&-1&0&0&0\\
    0&0&0&0&0&-1&2&-1&0&0\\
    0&0&0&0&0&0&-1&2&-1&0\\
    0&0&0&0&0&0&0&-1&2&-2\\
    0&0&0&0&0&0&0&0&-1&2\\
  \end{bmatrix},\\
B^{1}_{9}=
  \begin{bmatrix}
    2&-1&0&0&0&0&0&0&0\\
    -1&2&-1&0&0&0&0&0&0\\
    0&-1&2&-1&0&0&0&0\\
    0&0&-1&2&-1&0&0&0&0\\
    0&0&0&-1&2&-1&0&0&0\\
    0&0&0&0&-1&2&-1&0&0\\
    0&0&0&0&0&-1&2&-1&0\\
    0&0&0&0&0&0&-1&2&-2\\
    0&0&0&0&0&0&0&-1&2\\
  \end{bmatrix}.
\end{align*}
We look at the principal submatrix $B^{1}_{9}$. Since only hyperbolic type can have the affine type generalized Cartan matrices as principal submatrices while contained in a $N_{3}$ type generalized Cartan matrices by Corollary \ref{204}, we only consider hyperbolic types. In the list of dimension $10$ hyperbolic type generalized Cartan matrices, none contains it except $H_{2,79,10}$. If $a_{2,11}=a_{11,2}=-1$ in $B^{1}_{9}$, then notwithstanding what $a_{1,11}$ and $a_{11,1}$ are, there would be a hyperbolic type principal submatrix with dimension smaller than $10$. So $a_{i,11}=a_{11,i}=0$ for $i\in \{1,2,3,4,5,6,7,8,9\}$ in $B^{1}_{9}$. Now we think about the value of $a_{1,11}$ and $a_{11,1}$. All dimension $2$ principal submatrices should be affine or finite by Corollary \ref{204}. If it is not $0$, it is obvious that there would be an affine type like $E^{1}_{8}$, $A^{2}_{2}$, $A^{1}_{1}$, $A^{2}_{6}$ or $B^{1}_{6}$ with dimension smaller than $10$. By Corollary \ref{204}, it is impossible. Hence there are no $B^{1}_{11}$.

Now we consider the case of $(H_{2,78,10},D^{1}_{9})$,
\begin{align*}
H_{2,78,10}=
  \begin{bmatrix}
    2&-1&0&0&0&0&0&0&0&0\\
    -1&2&-1&0&0&0&0&0&0&0\\
    0&-1&2&-1&-1&0&0&0&0&0\\
    0&0&-1&2&0&0&0&0&0&0\\
    0&0&-1&0&2&-1&0&0&0&0\\
    0&0&0&0&-1&2&-1&0&0&0\\
    0&0&0&0&0&-1&2&0&-1&0\\
    0&0&0&0&0&0&0&2&-1&0\\
    0&0&0&0&0&0&-1&-1&2&-1\\
    0&0&0&0&0&0&0&0&-1&2\\
  \end{bmatrix},\\
D^{1}_{9}=
  \begin{bmatrix}
    2&-1&0&0&0&0&0&0&0\\
    -1&2&-1&0&0&0&0&0&0\\
    0&-1&2&-1&0&0&0&0\\
    0&0&-1&2&-1&0&0&0&0\\
    0&0&0&-1&2&-1&0&0&0\\
    0&0&0&0&-1&2&0&-1&0\\
    0&0&0&0&0&0&2&-1&0\\
    0&0&0&0&0&-1&-1&2&-1\\
    0&0&0&0&0&0&0&-1&2\\
  \end{bmatrix}.
\end{align*}
We look at the principal submatrix $D^{1}_{9}$. In the list of dimension $9$ hyperbolic type generalized Cartan matrices, none contains it except itself. If $a_{2,11}=a_{11,2}=-1$ in $D^{1}_{9}$, then notwithstanding what $(a_{1,11}/a_{11,1})$ is, there would be a hyperbolic type principal submatrices with dimension smaller than $10$. So $a_{i,11}=a_{11,i}=0$ for $i\in \{1,2,3,4,5,6,7,8,9\}$ in $D^{1}_{9}$. Now we think about the value of $a_{1,11}$ and $a_{11,1}$. All dimension $2$ principal submatrices should be affine or finite by Corollary \ref{204}. If it is not $0$, it is obvious that there would be an affine type like $E^{1}_{8}$, $A^{2}_{2}$, $A^{1}_{1}$, $A^{2}_{6}$, or $B^{1}_{6}$ with dimension smaller than $10$. By Corollary \ref{204}, it is impossible.

Now we consider the case of $(H_{2,77,10},E^{1}_{9})$, where

\begin{align*}
H_{2,77,10}=
  \begin{bmatrix}
    2&-1&0&0&0&0&0&0&0&0\\
    -1&2&-1&0&0&0&0&0&0&0\\
    0&-1&2&-1&0&0&0&0&0&0\\
    0&0&-1&2&-1&0&0&0&0&0\\
    0&0&0&-1&2&-1&0&0&0&0\\
    0&0&0&0&-1&2&-1&0&0&0\\
    0&0&0&0&0&-1&2&-1&-1&0\\
    0&0&0&0&0&0&-1&2&0&0\\
    0&0&0&0&0&0&-1&0&2&-1\\
    0&0&0&0&0&0&0&0&-1&2\\
  \end{bmatrix},\\
E^{1}_{9}=
  \begin{bmatrix}
    2&-1&0&0&0&0&0&0&0\\
    -1&2&-1&0&0&0&0&0&0\\
    0&-1&2&-1&0&0&0&0&0\\
    0&0&-1&2&-1&0&0&0&0\\
    0&0&0&-1&2&-1&0&0&0\\
    0&0&0&0&-1&2&-1&-1&0\\
    0&0&0&0&0&-1&2&0&0\\
    0&0&0&0&0&-1&0&2&-1\\
    0&0&0&0&0&0&0&-1&2\\
  \end{bmatrix}.
\end{align*}
We look at the principal submatrix $E^{1}_{9}$. In the list of dimension $9$ hyperbolic type generalized Cartan matrices, none contains it except $H_{2,77,10}$. If $a_{2,11}=a_{11,2}=-1$ in $E^{1}_{9}$, then notwithstanding what $a_{1,11}$ and $a_{11,1}$ are, there would be a hyperbolic type principal submatrix with dimension smaller than $10$. So $a_{i,11}=a_{11,i}=0$ for $i\in \{1,2,3,4,5,6,7,8,9\}$ in $E^{1}_{9}$. Now we think about the value of $a_{1,11}$ and $a_{11,1}$. All dimension $2$ principal submatrices should be affine or finite by Corollary \ref{204}. If it is not $0$, it is obvious that there would be an affine type like $E^{1}_{8}$, $A^{2}_{2}$, $A^{1}_{1}$, $A^{2}_{10}$, or $B^{1}_{10}$ with dimension smaller than $10$. By Corollary \ref{204}, it is impossible.

In conclusion, there are no $N_{3,11}$ type generalized Cartan matrices. Now we will show that the supremum is $10$.

Now we consider the case of $(H_{2,77,10},E^{1}_{9})$, where

\begin{align*}
H_{2,77,10}=
  \begin{bmatrix}
    2&-1&0&0&0&0&0&0&0&0\\
    -1&2&-1&0&0&0&0&0&0&0\\
    0&-1&2&-1&0&0&0&0&0&0\\
    0&0&-1&2&-1&0&0&0&0&0\\
    0&0&0&-1&2&-1&0&0&0&0\\
    0&0&0&0&-1&2&-1&0&0&0\\
    0&0&0&0&0&-1&2&-1&-1&0\\
    0&0&0&0&0&0&-1&2&0&0\\
    0&0&0&0&0&0&-1&0&2&-1\\
    0&0&0&0&0&0&0&0&-1&2\\
  \end{bmatrix},\\
E^{1}_{9}=
  \begin{bmatrix}
    2&-1&0&0&0&0&0&0&0\\
    -1&2&-1&0&0&0&0&0&0\\
    0&-1&2&-1&0&0&0&0&0\\
    0&0&-1&2&-1&0&0&0&0\\
    0&0&0&-1&2&-1&0&0&0\\
    0&0&0&0&-1&2&-1&-1&0\\
    0&0&0&0&0&-1&2&0&0\\
    0&0&0&0&0&-1&0&2&-1\\
    0&0&0&0&0&0&0&-1&2\\
  \end{bmatrix}.
\end{align*}
We look at the principal submatrix $E^{1}_{9}$. In the list of dimension $9$ hyperbolic type generalized Cartan matrices, none contains $E^{1}_{9}$ except $H_{2,77,10}$. If $a_{2,11}=a_{11,2}=-1$ in $E^{1}_{9}$, then notwithstanding what $a_{1,11}$ and $a_{11,1}$ are, there would be a hyperbolic type principal submatrix with dimension smaller than $10$. So $a_{i,11}=a_{11,i}=0$ for $i\in \{1,2,3,4,5,6,7,8,9\}$ in $E^{1}_{9}$. Now we think about the value of $a_{1,11}$ and $a_{11,1}$. All dimension $2$ principal submatrices should be affine or finite by Corollary \ref{204}. If it is not $0$, it is obvious that there would be an affine type like $E^{1}_{8}$, $A^{2}_{2}$, $A^{1}_{1}$, $A^{2}_{10}$, or $B^{1}_{10}$ with dimension smaller than $10$. By Corollary \ref{204}, it is impossible. In conclusion, there are no $N_{3,11}$ type generalized Cartan matrices.

Now we will show that the supremum is $10$. We start with the pair $(H_{2,74,9},E^1_{7})$, and construct a $N_{3}$ type generalized Cartan matrix $A_{1}$ such that the matrix $H_{2,74,9}$ is a principal submatrix of $A_{1}$. The matrices $H_{2,74,9},E^1_{7}$ are shown as follows
\begin{align*}
H_{2,74,9}=
    \begin{bmatrix}
    2&-1&0&0&0&0&0&0&0\\
    -1&2&-1&0&0&0&0&0&0\\
    0&-1&2&-1&0&0&0&0&0\\
    0&0&-1&2&-1&-1&0&0&0\\
    0&0&0&-1&2&0&0&0&0\\
    0&0&0&-1&0&2&-1&0&0\\
    0&0&0&0&0&-1&2&-1&0\\
    0&0&0&0&0&0&-1&2&-1\\
    0&0&0&0&0&0&0&-1&2\\
  \end{bmatrix},\\
  E^{1}_{7}=
  \begin{bmatrix}
    2&-1&0&0&0&0&0&0\\
    -1&2&-1&0&0&0&0&0\\
    0&-1&2&-1&-1&0&0&0\\
    0&0&-1&2&0&0&0&0\\
    0&0&-1&0&2&-1&0&0\\
    0&0&0&0&-1&2&-1&0\\
    0&0&0&0&0&-1&2&-1\\
    0&0&0&0&0&0&-1&2\\
  \end{bmatrix}.
\end{align*}
To begin with, we consider in hyperbolic type generalized Cartan matrices, $E^1_{7}$ is only contained in $H_{2,74,9}$. After adding one row and column, the possible affine type generalized Cartan matrices are $B^1_\ell$, $E^1_{7}$, $E^1_{6}$, $E^1_{9}$ $A^2_{2\ell-1}$, and $D^2_{\ell+1}$. To avoid having affine type principal submatrices of dimension smaller than $8$, we take $a_{9,10}=a_{10,9}=-1$, and $a_{i,10}=a_{10,i}=0$ for $i\in \{1,2,3,4,5,6,7,8,9\}$. Subsequently, we can have the following generalized Cartan matrix with dimension $10$. It is not affine, finite or hyperbolic, but every principal submatrix is either affine, finite or hyperbolic. More specifically, the highest dimension for affine type principal submatrices is $8$ and the dimension for hyperbolic type principal submatrix is $9$. Therefore, the matrix we constructed from $H_{2,74,9}$ is a $N_3$ type generalized Cartan matrix of dimension $10$ as follows.

\[
  \begin{bmatrix}
    2&-1&0&0&0&0&0&0&0&0\\
    -1&2&-1&0&0&0&0&0&0&0\\
    0&-1&2&-1&0&0&0&0&0&0\\
    0&0&-1&2&-1&-1&0&0&0&0\\
    0&0&0&-1&2&0&0&0&0&0\\
    0&0&0&-1&0&2&-1&0&0&0\\
    0&0&0&0&0&-1&2&-1&0&0\\
    0&0&0&0&0&0&-1&2&-1&0\\
    0&0&0&0&0&0&0&-1&2&-1\\
    0&0&0&0&0&0&0&0&-1&2\\
  \end{bmatrix}
\]

\bigskip
\noindent\small{\textsc{Cambridge A-Level center, Hangzhou Foreign Languages School}\\
309 Liuhe Rd, Xihu District, Hangzhou, Zhejiang Province, China}\\
\emph{E-mail address}:  \texttt{wendywangkh@outlook.com}


\begin{thebibliography}{12}
\bibitem{Carb}
Carbone, L., Chung, S., Cobbs, L., McRae, R., Nandi, D., Naqvi, Y., Penta, D. (2010). Classification of hyperbolic Dynkin diagrams, root lengths and Weyl group orbits. Journal of Physics A: Mathematical and Theoretical, 43(15), 155209.


\bibitem{Hum}
Humphreys, J. E. (2012). Introduction to Lie algebras and representation theory (Vol. 9). Springer Science and Business Media.

\bibitem{Kac}
Kac, V. G. (1990). Infinite-dimensional Lie algebras. Cambridge university press.

\bibitem{Sac}
Saclioglu, C. (1989). Dynkin diagrams for hyperbolic Kac-Moody algebras. Journal of Physics A: Mathematical and General, 22(18), 3753.
\end{thebibliography}
\end{document}